\documentclass[11 pt]{article}
\usepackage[T1]{fontenc}
\usepackage[utf8]{inputenc}
\usepackage[english]{babel}
\usepackage[margin=1.1in]{geometry}
\usepackage{microtype}
\usepackage{braket}
\usepackage{amsmath}
\usepackage{amssymb}
\usepackage{amsthm, aliascnt}
\usepackage{mathtools}
\usepackage{mathrsfs}
\usepackage{xcolor}
\usepackage{hyperref}
\usepackage{enumitem}
\usepackage{tikz}
\usepackage{hyperref}
\usepackage[maxbibnames=99]{biblatex}
\usepackage{csquotes}
\usepackage{esint}%
\usetikzlibrary{patterns}
\usepackage{tikz,graphicx}

\hypersetup{
                colorlinks=true,
                linkcolor={magenta},
                citecolor={cyan},}

\theoremstyle{plain}
\newtheorem{teor}{Theorem}
\numberwithin{teor}{section}
\numberwithin{equation}{section}
\theoremstyle{definition}
\newaliascnt{defi}{teor}
\newtheorem{defi}[defi]{Definition}
\aliascntresetthe{defi}

\theoremstyle{plain}
\newaliascnt{lemma}{teor}%
\newtheorem{lemma}[lemma]{Lemma}
\aliascntresetthe{lemma}
\theoremstyle{plain}
\newaliascnt{prop}{teor}%
\newtheorem{prop}[prop]{Proposition}
\aliascntresetthe{prop}
\theoremstyle{plain}
\newaliascnt{cor}{teor}%

\aliascntresetthe{cor}
\theoremstyle{definition}
\newaliascnt{ex}{teor}%

\aliascntresetthe{ex}
\theoremstyle{remark}
\newaliascnt{oss}{teor}%
\newtheorem{oss}[oss]{Remark}
\aliascntresetthe{oss}
\DeclarePairedDelimiter{\abs}{\lvert}{\rvert}

\DeclareMathOperator{\capac}{Cap}
\newcommand{\marcomm}[1]{\marginpar{\begin{flushright}#1\end{flushright}}}%
\newcommand{\R}{\mathbb{R}}

\newcommand{\Ln}{\mathcal{L}^n}
\newcommand{\Hn}{\mathcal{H}^{n-1}}
\newcommand{\test}{\frac{\abs{\nabla u}}{u}}
\newcommand{\testt}{\frac{\abs{\nabla u^*}}{u^*}}
\newcommand{\eps}{\varepsilon}

\DeclareMathOperator{\divv}{div}

\makeatletter
\newcommand{\leqnomode}{\tagsleft@true\let\veqno\@@leqno}
\newcommand{\reqnomode}{\tagsleft@false\let\veqno\@@eqno}

\definecolor{myblue}{HTML}{427eba}
\definecolor{myred}{HTML}{cc5653}
\definecolor{myyellow}{HTML}{fa8b30}

\newcommand{\Addresses}{{%
 \bigskip 
 \footnotesize 
 
 \textsc{Dipartimento di Matematica e Applicazioni ``R. Caccioppoli'', Universit\`a degli studi di Napoli Federico II, Via Cintia, Complesso Universitario Monte S. Angelo, 80126 Napoli, Italy.}\par\nopagebreak 
 
 \medskip 
 
 \textit{E-mail address}, P.~Acampora: \texttt{paolo.acampora@unina.it} 
 \medskip 
 
\textsc{Mathematical and Physical Sciences for Advanced Materials and Technologies, Scuola Superiore Meridionale, Largo San Marcellino 10, 80126, Napoli, Italy.}\par\nopagebreak 
 
 \medskip 
 
 \textit{E-mail address}, E.~Cristoforoni: \texttt{emanuele.cristoforoni@unina.it} 
}} 
\makeatother

\title{An isoperimetric result for an energy related to the $p$-capacity}
\author{P. Acampora, E. Cristoforoni }
\date{}

\begin{document}
\reversemarginpar
\maketitle
\begin{abstract}
    In this paper, we generalize the notion of relative $p$-capacity of $K$ with respect to $\Omega$, by replacing the Dirichlet boundary condition with a Robin one. We show that, under volume constraints, our notion of $p$-capacity is minimal when $K$ and $\Omega$ are concentric balls. We use the $H$-function (see \cite{bossel, daners}) and a derearrangement technique.
    
 
\end{abstract}
\section{Introduction}
Let $p>1$, $\beta>0$ be real numbers. For every open bounded sets $\Omega\subset\R^n$ with Lipschitz boundary, and every compact set $K\subseteq \overline{\Omega}$ with Lipschitz boundary, we define
\begin{equation}
\label{problema0}
E_{\beta,p}(K,\Omega)=\inf_{\substack{v\in W^{1,p}(\Omega)\\ v=1 \text{ in } K}}\left(\int_\Omega \abs{\nabla v}^p\,dx+\beta\int_{\partial\Omega} \abs{v}^p\,d\Hn\right).
\end{equation}
We notice that it is sufficient to minimize among all functions $v\in H^1(\Omega)$ with $v=1$ in $K$ and $0\le v\le 1$ a.e., moreover if $K,\Omega$ are sufficiently smooth, a minimizer $u$ satisfies
\begin{equation}
\label{eq: pde}
\begin{cases}
u=1 &\text{in }K,\\[5 pt]
\Delta_p u=0 &\text{in }\Omega \setminus K,\\[6 pt]
\abs{\nabla u}^{p-2}\dfrac{\partial u}{\partial \nu}+\beta\abs{u}^{p-2}u=0 &\text{on }\partial \Omega\setminus \partial K,
\end{cases}
\end{equation}
where $\Delta_p u =\divv\left(\abs{\nabla u}^{p-2} \nabla u\right)$ is the $p$-Laplacian of $u$ and $\nu$ is the outer unit normal to $\partial\Omega$. If $\mathring{K}=\Omega$, equation \eqref{eq: pde} has to be intended as $u=1$ in $\Omega$, and the energy is 
\[
E_{\beta, p}(\Omega,\Omega) = \beta \Hn(\partial\Omega).
\]
In general, equation \eqref{eq: pde} has to be intended in the weak sense, that is: for every $\varphi\in W^{1,p}(\Omega)$ such that $\varphi\equiv 0$ in $K$, 
\begin{equation}
    \label{eq: EL}
    \int_\Omega \abs{\nabla u}^{p-2}\nabla u \nabla \varphi\, d\Ln + \beta\int_{\partial\Omega}u^{p-1}\varphi\,d\Hn= 0.
\end{equation}
In particular if $u$ is a minimizer, letting $\varphi=u-1$, we have that
\[E_{\beta,p}(K,\Omega)=\int_\Omega \abs{\nabla u}^p\,dx+\beta\int_{\partial\Omega} u^p\,d\Hn=\beta\int_{\partial\Omega} u^{p-1}\,d\Hn.\]
Moreover from the strict convexity of the functional, the minimizer is the unique solution to \eqref{eq: EL}.\medskip

This problem is related to the so-called \emph{relative $p$-capacity of $K$ with respect to $\Omega$}, defined as
\[
\capac_p(K,\Omega):=\inf_{\substack{v\in W^{1,p}_0(\Omega)\\ v=1 \text{ in } K}}\left(\int_\Omega \abs{\nabla v}^p\,dx\right).
\]
In the case $p=2$ it represents the electrostatic capacity of an annular condenser consisting of a conducting surface $\partial\Omega$, and a conductor $K$, where the electrostatic potential is prescribed to be 1 inside $K$ and 0 outside $\Omega$. 
Let $\omega_n$ be the measure of the unit sphere in $\R^n$, and let $M>\omega_n$, then it is well known that there exists some $r\ge1$ such that
\[
\min_{\substack{\abs{K}=\omega_n\\ \abs{\Omega}\le M}}\capac_p(K,\Omega)=\capac_p(B_1,B_r).
\]
This is an immediate consequence of the Pólya-Szegő inequality for the Schwarz rearrangement (see for instance \cite{polya, kesavan}). We are interested in studying the same problem for the energy defined in \eqref{problema0}, which corresponds to changing the Dirichlet boundary condition on $\partial\Omega$ into a Robin boundary condition, namely, we consider the following problem
\begin{equation}\label{problema} \inf_{\substack{\abs{K}=\omega_n\\ \abs{\Omega}\le M}} E_{\beta,p}(K,\Omega).\end{equation}
In this case, the previous symmetrization techniques cannot be employed anymore.\medskip

Problem \eqref{problema} has been studied in the linear case $p=2$ in \cite{nahon}, with more general boundary conditions on $\partial\Omega$, namely
\[
\frac{\partial u}{\partial \nu}+\frac{1}{2}\Theta'(u)=0,
\]
where $\Theta$ is a suitable increasing function vanishing at $0$. This problem has been addressed to thermal insulation (see for instance \cite{CK, AC}). Our main result reads as follows.

\begin{teor}\label{teorema}
Let $\beta>0$ such that
\[
\beta^{\frac{1}{p-1}}>\frac{n-p}{p-1}.
\]
Then, for every $M>\omega_n$ the solution to problem \eqref{problema} is given by two concentric balls $(B_1,B_r)$, that is 
\[
\min_{\substack{\abs{K}=\omega_n\\ \abs{\Omega}\le M}} E_{\beta,p}(K,\Omega) = E_{\beta, p}(B_1, B_r),
\]
in particular we have that either $r=1$ or $M=\omega_n r^n$. 

Moreover, if $K_0\subseteq\overline{\Omega_0}$ are such that 
\[
E_{\beta,p}(K_0,\Omega_0)=\min_{\substack{\abs{K}=\omega_n\\ \abs{\Omega}\le M}} E_{\beta,p}(K,\Omega),
\]
and $u$ is the minimizer of $E_{\beta,p}(K_0,\Omega_0)$, then the sets $\Set{u=1}$ and $\Set{u>0}$ coincide with two concentric balls up to a $\Hn$-negligible set.
\end{teor}
\begin{oss}
In the case
\[
\beta^{\frac{1}{p-1}}\le \frac{n-p}{p-1},
\]
adapting the symmetrization techniques used in \cite{nahon}, it can be proved that a solution to problem \eqref{problema} is always given by the pair $(B_1,B_1)$.
\end{oss}
We point out that the proof of the theorem relies on the techniques involving the $H$-function introduced in \cite{bossel, daners}.\medskip

The case in which $\Omega$ is the Minkowski sum $\Omega=K+B_r(0)$, the energy $E_{\beta,p}(K,\Omega)$, has been studied in \cite{Rossella} under suitable geometrical constraints.
\section{Proof of the theorem}
In order to prove \autoref{teorema}, we start by studying the function 
\[R\mapsto E_{\beta,p}(B_1,B_R).\]
A similar study of the previous function can also be found in \cite{Rossella}. Let 
\[
\Phi_{p,n}(\rho)=\begin{cases} \log(\rho) &\text{if } p=n,\\[7 pt] -\dfrac{p-1}{n-p}\dfrac{1}{\rho^\frac{n-p}{p-1}} &\text{if }p\ne n. \end{cases}\]
For every $R>1$, consider
\begin{equation}
\label{eq: ustar}
u^*(x)=1-\dfrac{\beta^{\frac{1}{p-1}}\left(\Phi_{p,n}(\abs{x})-\Phi_{p,n}(1)\right)_+}{\Phi'_{p,n}(R)+\beta^\frac{1}{p-1}\left(\Phi_{p,n}(R)-\Phi_{p,n}(1)\right)},
\end{equation}
the solution to 
\[\begin{cases}u^*=1 &\text{in }B_1,\\[5 pt]
\Delta_p u^*=0 &\text{in }B_R\setminus B_1,\\[6 pt]
\abs{\nabla u^*}^{p-2}\dfrac{\partial u^*}{\partial \nu}+\beta\abs{u^*}^{p-2}u^*=0 &\text{on }\partial B_R.\end{cases}\]
We have that 
\begin{equation}
\label{eq: Ebetarho}
\begin{split}E_{\beta,p}(B_1,B_R)&=\int_{B_R} \abs{\nabla u^*}^p\,dx+\beta\int_{\partial B_R} \abs{u^*}^p\,d\Hn\\[10 pt]&=\dfrac{n\omega_n\beta}{\left[\Phi'_{p,n}(R)+\beta^\frac{1}{p-1}\left(\Phi_{p,n}(R)-\Phi_{p,n}(1)\right)\right]^{p-1}}.\end{split}
\end{equation}
Notice that $E_{\beta,p}(B_1,B_R)$ is decreasing in $R>0$ if and only if
\[\dfrac{d }{d R} \left(\Phi'_{p,n}(R)+\beta^{\frac{1}{p-1}}\Phi_{p,n}(R)\right)\ge 0\]
that is, if and only if
\[R\ge \dfrac{n-1}{p-1}\dfrac{1}{\beta^\frac{1}{p-1}}=:\alpha_{\beta,p}.\]
Moreover 
\[\begin{split}&E_{\beta,p}(B_1,B_1)=n\omega_n \beta,\\[10 pt]
\lim_{R\to\infty}&E_{\beta,p}(B_1,B_R)=\begin{cases} n\omega_n \left(\dfrac{n-p}{p-1}\right)^{p-1} &\text{if }p<n,\\[10 pt]
0 &\text{if }p\ge n.
\end{cases}\end{split}\]
Therefore, there are three cases:
\begin{itemize}
    \item {if \[\beta^{\frac{1}{p-1}}\ge \dfrac{n-1}{p-1},\] $R\in[1,+\infty)\mapsto E_{\beta,p}(B_1,B_R)$ is decreasing; }
    \item {if \[\dfrac{n-p}{p-1}<\beta^{\frac{1}{p-1}}<\dfrac{n-1}{p-1},\] $R\in[1,+\infty)\mapsto E_{\beta,p}(B_1,B_R)$ increases on $[1,\alpha_{\beta,p}]$ and decreases on $[\alpha_{\beta,p},+\infty)$, with the existence of a unique $R_{\beta,p}>\alpha_{\beta,p}$ such that $E_{\beta,p}(B_1,B_{R_{\beta,p}})=E_{\beta,p}(B_1,B_1)$; }
    \item {if \[\beta^{\frac{1}{p-1}}\le \dfrac{n-p}{p-1},\] $R\in[1,+\infty)\mapsto E_{\beta,p}(B_1,B_R)$ reaches its minimum at $R=1$. }
\end{itemize}
See for instance \autoref{fig}, where
\[
\beta_1=\left(\frac{n-p}{p-1}\right)^{p-1}, \qquad \beta_2=\left(\frac{n-1}{p-1}\right)^{p-1}, \qquad p=2.5, \qquad n=3.
\]
\begin{figure}
\begin{tikzpicture}
    \node[anchor=south west,inner sep=0] (image) at (0,0) {\includegraphics[width=.50\linewidth]{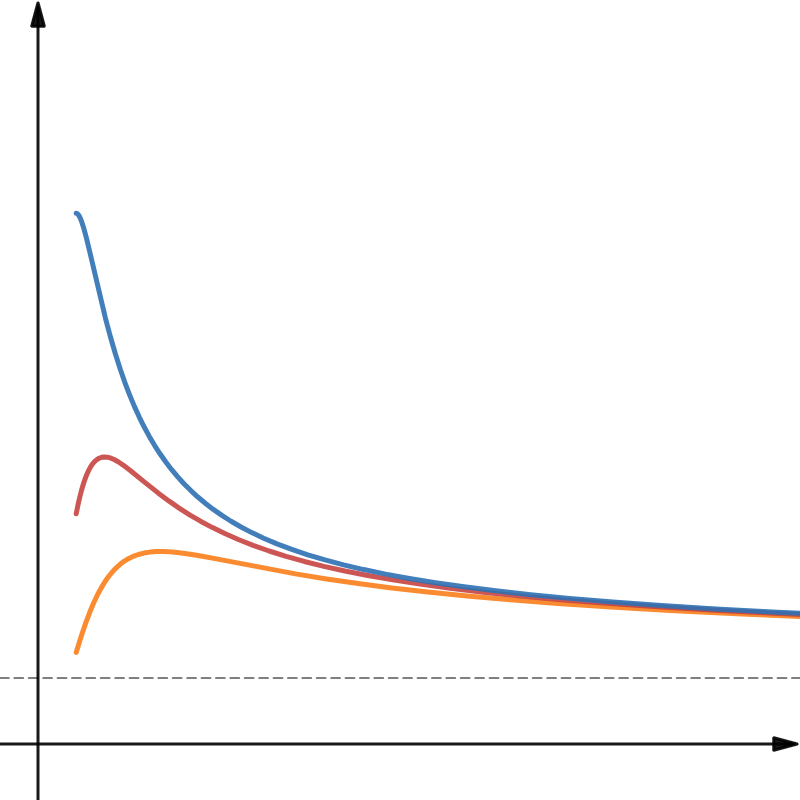}};
    \begin{scope}[x={(image.south east)},y={(image.north west)}]
        \node[anchor=west] at (0.05,0.95) {$E_{\beta,p}(B_1,B_r)$};
        \node[anchor=north] at (0.95,0.05) {$r$};
        \node[anchor=east] at (0,0.55) {\textcolor{myblue}{$\beta\ge\beta_2$}};
        \node[anchor=east] at (0,0.4) {\textcolor{myred}{$\beta_1<\beta<\beta_2$}};
        \node[anchor=east] at (0,0.25) {\textcolor{myyellow}{$\beta\le\beta_1$}};
    \end{scope}
\end{tikzpicture}
\caption{$E_{\beta,p}(B_1,B_r)$ depending on the value of $\beta$}
\label{fig}
\end{figure}
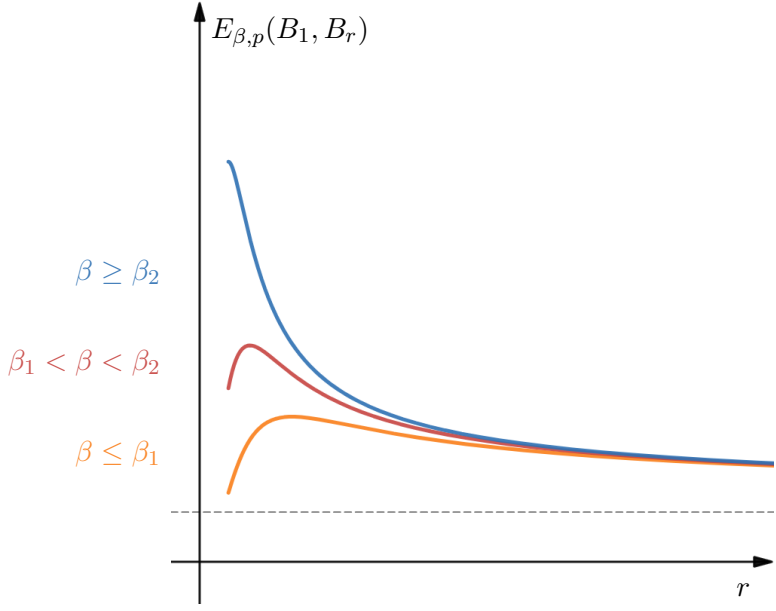

In the following, we will need
\begin{lemma}\label{lemma5}
Let $R>1$,$\beta>0$ and let $u^*$ be the solution of the problem on $(B_1,B_R)$. Then 
\[\dfrac{\abs{\nabla u^*}}{u^*}\le \beta^{\frac{1}{p-1}}\] 
in $B_R\setminus B_1$, if and only if
\[E_{\beta,p}(B_1,B_\rho)\ge E_{\beta,p}(B_1,B_R)\]
for every $\rho\in[1,R]$.
\end{lemma}
\begin{proof}
Recalling the expressions of $u^*$ in \eqref{eq: ustar}, by straightforward computations we have that
\[\dfrac{\abs{\nabla u^*}}{u^*}\le \beta^{\frac{1}{p-1}}\] 
in $B_R\setminus B_1$ if and only if
\begin{equation}
\label{eq: monotonicity}
\Phi'_{p,n}(R)+\beta^\frac{1}{p-1}\left(\Phi_{p,n}(R)-\Phi_{p,n}(1)\right)\ge \Phi'_{p,n}(\rho)+\beta^\frac{1}{p-1}\left(\Phi_{p,n}(\rho)-\Phi_{p,n}(1)\right)
\end{equation}
for every $\rho\in[1,R]$, using the expression of $E_{\beta,p}(B_1,B_\rho)$ in \eqref{eq: Ebetarho}, \eqref{eq: monotonicity} is equivalent to
\[E_{\beta,p}(B_1,B_\rho)\ge E_{\beta,p}(B_1,B_R)\]
for every $\rho\in[1,R]$.
\end{proof}

\begin{defi}
Let $\Omega\subseteq\R^n$ be an open set, and let $U\subseteq\Omega$ be another set. We define the \emph{internal boundary} of $U$ as 
\[
\partial_i U=\partial U\cap \Omega,
\]
and the \emph{external boundary} of $U$ as
\[
\partial_e U= \partial U\cap\partial\Omega.
\]
\end{defi}
Let $K\subseteq\overline{\Omega}\subseteq\R^n$ be open bounded sets, and let $u$ be the minimizer of $E_{\beta.p}(K,\Omega)$. In the following, we denote by
\[
U_t=\Set{x\in\Omega | u(x)>t}.
\]
\begin{defi}[$H$-function]
\label{defi: H}
Let $\varphi\in W^{1,p}(\Omega)$. We define
\[
H(t,\varphi)=\int_{\partial_i U_t}\abs{\varphi}^{p-1}\,d\Hn - (p-1)\int_{U_t}\abs{\varphi}^p\,d\Ln+\beta\Hn(\partial_e U_t).
\]
\end{defi}
Notice that this definition is slightly different from the one given in \cite{bucdan}.
\begin{lemma}
\label{lemma1}
Let $K\subseteq\Omega\subseteq\R^n$ be an open, bounded sets, and let $u$ be the minimizer of $E_{\beta,p}(K,\Omega)$. Then for a.e. $t\in(0,1)$ we have
\[
H\left(t,\test\right)=E_{\beta,p}(K,\Omega).
\]
\begin{proof}
Recall that
\begin{equation}
    \label{eq: Hfunction1}
E_{\beta,p}(K,\Omega)=\int_{\Omega}\abs{\nabla u}^p\,d\Ln + \beta\int_{\partial\Omega}u^p=\beta\int_{\partial\Omega}u^{p-1}\,d\Hn.
\end{equation}
Let $t\in (0,1)$, we construct the following test functions: let $\varepsilon>0$, and let
\[
\varphi_\eps(x)=\begin{dcases}
-1 &\text{if }u(x)\le t, \\
\frac{u(x)-t}{\eps u(x)^{p-1}}-1 &\text{if }t<u(x)\le t+\eps, \\
\frac{1}{u(x)^{p-1}}-1 &\text{if } u(x)>t+\eps,
\end{dcases}
\]
so that $\varphi_\eps$ is an approximation the function $(u^{1-p}\chi_{U_t}-1)$, and
\[
\nabla \varphi_\eps(x)=\begin{dcases}
0 &\text{if }u(x)\le t, \\
\frac{1}{\eps}\left(\frac{\nabla u(x)}{u(x)^{p-1}}-(p-1)\frac{\nabla u(x) (u(x)-t)}{u(x)^p}\right) &\text{if }t<u(x)\le t+\eps, \\
-(p-1)\frac{\nabla u(x)}{u(x)^{p}} &\text{if } u(x)>t+\eps.
\end{dcases}
\]
We have that $\varphi_\eps$ is an admissible test function for the Euler-Lagrange equation \eqref{eq: EL}, which entails
\[
\begin{split}
0=&\frac{1}{\eps}\int_{\{t<u\le t+\eps\}\cap \Omega}\!\frac{\abs{\nabla u}^{p-1}}{u^{p-1}}\abs{\nabla u}\, d\Ln-(p-1)\int_{\{t<u\le t+\eps\}\cap \Omega}\!\frac{\abs{\nabla u}^{p}}{u^{p}}\frac{u-t}{\eps}\, d\Ln \\[7 pt]
&-(p-1)\int_{\{u>t+\eps\}\cap\Omega}\frac{\abs{\nabla u}^{p}}{u^{p}}\,d\Ln+\beta\int_{\{t<u\le t+\eps\}\cap \partial\Omega}\frac{u-t}{\eps}\,d\Hn \\[7 pt]
&+\beta\Hn\left(\partial\Omega\cap \{u>t+\eps\}\right)-\beta\int_{\partial \Omega}u^{p-1}\,d\Hn.
\end{split}
\]
Letting now $\eps$ go to 0, by coarea formula we get that for a.e. $t\in (0,1)$
\begin{equation}
\label{eq: Hfunction2}
\begin{split}
\beta\int_{\partial \Omega}u^{p-1}\,d\Hn=&\int_{\partial_i U_t}\left(\frac{\abs{\nabla u}}{u}\right)^{p-1}\,d\Hn-(p-1)\int_{U_t}\left(\frac{\abs{\nabla u}}{u}\right)^{p}\,d\Ln\\[7 pt]
&+\beta \Hn(\partial_e U_t).
\end{split}
\end{equation}
Joining \eqref{eq: Hfunction1} and \eqref{eq: Hfunction2}, the lemma is proven.
\end{proof}
\end{lemma}
\begin{oss}
Notice that if $K,$ and $\Omega$ are two concentric balls, the minimizer $u$ is the one written in \eqref{eq: ustar}, for which the statement of the above Lemma holds for every $t\in(0,1)$.
\end{oss}
\begin{lemma}
\label{lemma2}
Let $\varphi\in L^\infty(\Omega)$. Then there exists $t\in(0,1)$ such that 
\[
H(t,\varphi)\le E_{\beta, p}(K,\Omega).
\]
\begin{proof}
Let 
\[
w=\abs{\varphi}^{p-1}-\left(\test\right)^{p-1},
\]
then we evaluate
\[
\begin{split}
H(t,\varphi)-H\left(t,\test\right)\!&= \int_{\partial_i U_t}\!\!w\,d\Hn -(p-1)\!\int_{U_t}\!\left(\abs{\varphi}^p-\left(\test\right)^p\right)\,d\Ln\\[7 pt]
&\le \int_{\partial_i U_t} w\,d\Hn-p\int_{U_t}\test w\,d\Ln\\[7 pt]&=-\frac{1}{t^{p-1}}\frac{d}{dt}\left(t^p\int_{U_t}\test w\,d\Ln \right),
\end{split}
\]
where we used the inequality
\begin{equation}
\label{eq: convexity}
a^p-b^p\le \frac{p}{p-1}a\,(a^{p-1}-b^{p-1}) \qquad \forall a,b\ge0.
\end{equation}
Multiplying by $t^{p-1}$ and integrating, we get
\begin{equation}
\label{eq: Hineq}
\int_0^1 t^{p-1}\left(H(t,\varphi)-H\left(t,\test\right)\right)\,dt\le-\left[t^p\int_{U_t}\test w\, d\Ln\right]_0^1=0,
\end{equation}
from which we obtain the conclusion of the proof.
\end{proof}
\end{lemma}
\begin{oss}
\label{rem: unicità}
    Notice that the inequality \eqref{eq: convexity} holds as equality if and only if $a=b$. Therefore, if $\varphi\ne\frac{\abs{\nabla u}}{u}$ on a set of positive measure, then the inequality in \eqref{eq: Hineq} is strict, since 
    \[
    \abs*{\Set{\varphi\ne \frac{\nabla u}{u}}\cap U_t}>0
    \]
    for small enough $t$. Therefore, there exists $S\subset(0,1)$ such that $\mathcal{L}^1(S)>0$ and for every $t\in S$ 
    \[
    H(t,\varphi)<E_{\beta,p}(K,\Omega).
    \]
\end{oss}
    In the following, we fix a radius $R$ such that $\abs{B_R}\ge\abs{\Omega}$, $u^*$ the minimizer of $E_{\beta,p}(B_1,B_R)$, and
\[
\begin{split}
H^*(t,\varphi)=&\int_{\partial\set{u^*>t}\cap B_R}\abs{\varphi}^{p-1}\,d\Hn - (p-1)\int_{\set{u^*>t}}\abs{\varphi}^p\,d\Ln\\[7 pt]
&+\beta\Hn(\partial\set{u^*<t}\cap\partial B_R).
\end{split}
\]
\begin{prop}
\label{teorema part1}
Let $\beta>0$. Assume that 
\begin{equation}
    \label{eq: ipotesisuphi}
\testt\le \beta^{\frac{1}{p-1}}.
\end{equation}
Then we have that
\[
E_{\beta,p}(K,\Omega)\ge E_{\beta,p}(B_1,B_R).
\]
\end{prop}
\begin{proof}
In the following, if $v$ is a radial function on $B_R$ and $r\in(0,R)$, we denote with abuse of notation
\[
v(r)=v(x),
\]
where $x$ is any point on $\partial B_r$. By \autoref{lemma1} we know that for every $t\in(0,1)$
\begin{equation}
\label{eq: theqball}
H^*\left(t,\testt\right)=E_{\beta,p}(B_1,B_R),
\end{equation}
while by \autoref{lemma2}, for every $\varphi\in L^{\infty}(\Omega)$ there exists a $t\in(0,1)$ such that
\begin{equation}
\label{eq: thineqgen}
E_{\beta,p}(K,\Omega)\ge H(t,\varphi).
\end{equation}
We aim to find a suitable $\varphi$ such that, for some $t$,
\begin{equation}
\label{eq: thineqgenball}
H(t,\varphi)\ge H^*\left(t,\testt\right),
\end{equation}
so that combining \eqref{eq: thineqgen}, \eqref{eq: thineqgenball}, and \eqref{eq: theqball} we conclude the proof. In order to construct $\varphi$, for every $t\in(0,1)$ we define
\begin{equation}
\label{eq: radius}
r(t)=\left(\frac{\abs{U_t}}{\omega_n}\right)^{\frac{1}{n}},
\end{equation}
then we set, for every $x\in\Omega$,
\[
\varphi(x)=\testt(r(u(x))).
\]
\marcomm{\textbf{Claim}} The functions $\varphi\chi_{U_t}$ and $\testt\chi_{B_{r(t)}}$ are equi-measurable, in particular
\begin{equation}
\label{eq: equi-meas}
    \int_{U_t}\varphi^p\,d\Ln=\int_{B_{r(t)}}\left(\testt\right)^{p}\,d\Ln.
\end{equation}
    Indeed, let $g(r)=\testt(r)$, and by
    coarea formula,
\begin{equation}
\label{eq: measlevset}
\begin{split}
\abs{{U_t\cap\set{\varphi>s}}}&=\!\int_{U_t\cap\Set{g(r(u(x)))>s}}\,d\Ln \\[7 pt]
&=\!\int_t^{+\infty}\int_{\partial^* U_{\tau}\cap\set{g(r(\tau))>s}}\frac{1}{\abs{\nabla u(x)}}\,d\Hn(x)\,d\tau \\[7 pt]
&=\!\int_0^{r(t)}\!\!\int_{\partial^* U_{r^{-1}(\sigma)}}\frac{1}{\abs{\nabla u(x)}\abs{r'(r^{-1}(\sigma))}}\chi_{\set{g(\sigma)>s}}\,d\Hn(x)\,d\sigma.
\end{split}
\end{equation}
Notice now that, since
\[
\omega_n r(\tau)^n=\abs{U_\tau},
\]
then
\begin{equation}
\label{eq: r'}
r'(\tau)=-\frac{1}{n\omega_n r(\tau)^{n-1}}\int_{\partial^* U_\tau}\frac{1}{\abs{\nabla u(x)}}\,d\Hn(x).
\end{equation}
Therefore, substituting in \eqref{eq: measlevset}, we get
\[
\abs{{U_t\cap\set{\varphi>s}}}=\int_0^{r(t)}n\omega_n \sigma^{n-1}\chi_{\Set{g(\sigma)>s}}\,d\sigma=\abs*{B_{r(t)}\cap\Set{\testt>s}};
\]
where we have used polar coordinates to get the last equality. Thus, the claim is proved.

Recalling the definition of $\varphi$, \eqref{eq: ipotesisuphi} reads
\[
\beta\ge \varphi^{p-1},
\]
then using \eqref{eq: equi-meas} and the definition of $H$ (see \autoref{defi: H}), we have
\begin{equation}
\label{eq: finalestimate}
\begin{split}
H(t,\varphi)&=\beta\Hn(\partial_e U_t)+\int_{\partial_i U_t}\varphi^{p-1}\,d\Hn-(p-1)\int_{U_t}\varphi^p\,d\Ln \\[7 pt]
&\ge \int_{\partial U_t}\varphi^{p-1}\,d\Hn-(p-1)\int_{B_{r(t)}}\left(\testt\right)^{p}\,d\Ln \\[7 pt]
&\ge \int_{\partial B_{r(t)}}\left(\testt\right)^{p-1}\,d\Hn -(p-1)\int_{B_{r(t)}}\left(\testt\right)^{p}\,d\Ln\\[7 pt]
&=H^*\left(u^*(r(t)),\testt\right)\\[7 pt]
&=E_{\beta,p}(B_1,B_R),
\end{split}
\end{equation}
where in the last inequality we have used the isoperimetric inequality and the fact that $\varphi$ is constant on $\partial U_t$.
\end{proof}
\begin{oss}
\label{rem: unicità2}
By \autoref{rem: unicità}, we have that if $K$ and $\Omega$ are such that
\[
E_{\beta,p}(K,\Omega)=E_{\beta, p}(B_1,B_R),
\]
then
\[
\varphi=\frac{\abs{\nabla u}}{u} \qquad \text{for a. e. }x\in\Omega,
\]
so that, by \autoref{lemma1}, we have equality in \eqref{eq: finalestimate} for a.e. $t\in(0,1)$. Thus, by the rigidity of the isoperimetric inequality, we get that $U_t$ coincides with a ball up to a $\Hn$-negligible set for a.e. $t\in(0,1)$. In particular, $\Set{u>0}=\bigcup_t U_t$ and $\set{u=1}=\bigcap_t U_t$ coincide with two balls up to a $\Hn$-negligible set.
\end{oss}
\begin{proof}[Proof of \autoref{teorema}]
Fix $M=\omega_n R^n$ with $R>1$. We divide the proof of the minimality of balls into two cases, and subsequently, we study the equality case. \medskip

Let us assume that
\[\beta^{\frac{1}{p-1}}\ge\dfrac{n-1}{p-1},\]
and recall that in this case the function
\[\rho\in[1,+\infty)\mapsto E_{\beta,p}(B_1,B_\rho)\]
is decreasing. Let $u^*$ be the minimizer of $E_{\beta,p}(B_1,B_R)$, by \autoref{lemma5} condition \eqref{eq: ipotesisuphi} holds and, by \autoref{teorema part1}, we have that a solution to \eqref{problema} is given by the concentric balls $(B_1,B_R)$.\medskip

Assume now that
\[\dfrac{n-p}{p-1}<\beta^{\frac{1}{p-1}}<\dfrac{n-1}{p-1},\]
then, in this case, letting \[\alpha_{\beta,p}=\dfrac{(n-1)}{(p-1)\beta^\frac{1}{p-1}},\] the function
\[\rho\in[1,+\infty)\mapsto E_{\beta,p}(B_1,B_\rho)\]
increases on $[1,\alpha_{\beta,p}]$ and decreases on $[\alpha_{\beta,p},+\infty)$, and there exist a unique $R_{\beta,p}>\alpha_{\beta,p}$ such that $E_{\beta,p}(B_1,B_{R_{\beta,p}})=E_{\beta,p}(B_1,B_1)$. If $R\ge R_{\beta,p}$ the function $u^*$, minimizer of $E_{\beta,p}(B_1,B_R)$, still satisfies condition \eqref{eq: ipotesisuphi} and, as in the previous case, a solution to \eqref{problema} is given by the concentric balls $(B_1,B_R)$. On the other hand, if $R<R_{\beta,p}$, we can consider $u^*_{\beta,p}$ the minimizer of $E_{\beta,p}(B_1,B_{R_{\beta,p}})$. By \autoref{lemma5} we have that, for the function $u^*_{\beta,p}$, condition \eqref{eq: ipotesisuphi} holds and, by \autoref{teorema part1}, we have that if $K$ and $\Omega$ are open bounded Lipschitz sets with $K\subseteq \Omega$, $\abs{K}=\omega_n$, and $\abs{\Omega}\le M$, then 
\[
E_{\beta,p}(K,\Omega)\ge E_{\beta,p}(B_1,B_{R_{\beta,p}})=E_{\beta,p}(B_1,B_1)
\]
and a solution to \eqref{problema} is given by the pair $(B_1,B_1)$. 

\medskip
For what concerns the equality case, we will follow the outline of the rigidity problem given in \cite[Section 3]{masiellopaoli} (see also \cite[Section 2]{arontalenti}). Let $K_0\subseteq\overline{\Omega_0}$ be such that 
\[
E_{\beta,p}(K_0,\Omega_0)=\min_{\substack{\abs{K}=\omega_n\\ \abs{\Omega}\le M}}E_{\beta,p}(K,\Omega),
\]
let $u$ be the minimizer of $E_{\beta,p}(K_0,\Omega_0)$. If $\mathring{K_0}=\Omega_0$, then $\abs{\Omega_0}=\abs{B_1}$  and isoperimetric inequality yields
\[\Hn(\partial\Omega_0)\ge\Hn(\partial B_1),\]
while, from the minimality of $(K_0,\Omega_0)$ we have that
\[E_{\beta,p}(K_0,\Omega_0)=\beta \Hn(\partial\Omega_0)\le E_{\beta,p}(B_1,B_1)= \beta\Hn(\partial B_1),\]
so that $\Hn(\Omega_0)=\Hn(\partial B_1)$. Hence, by the rigidity of the isoperimetric inequality we have that $\mathring{K_0}=\Omega_0$ are balls of radius $1$. On the other hand, if $\mathring{K_0}\ne\Omega_0$, from the first part of the proof, there exists $R_0>1$ such that $\abs{B_{R_0}}\ge M$ and
\[
E_{\beta,p}(K_0,\Omega_0)=E_{\beta,p}(B_1,B_{R_0}).
\] 
Therefore, by \autoref{rem: unicità2}, we have that for a.e. $t\in(0,1)$, the superlevel sets $U_t$ coincide with balls up to $\Hn$-negligible sets, and $\Set{u=1}$ and $\Set{u>0}$ coincide with balls, up to $\Hn$-negligible sets, as well. We only have to show that $\Set{u=1}$ and $\Set{u>0}$ are concentric balls. To this aim, let us denote by $x(t)$ the center of the ball $U_t$ and by $r(t)$ the radius of $U_t$, as already done in \eqref{eq: radius}. In addition, we also have that
\[
\frac{\abs{\nabla u^*}}{u^*}\Big(r\big(u(x)\big)\Big)=\varphi(x)=\frac{\abs{\nabla u}}{u}(x),
\]
so that, if $u(x)=t$, then $\abs{\nabla u(x)}=C_t>0$. This ensures that we can write
\[
\begin{split}
x(t)&=\frac{1}{\abs{U_t}}\int_{U_t}x\,d\Ln(x)\\[7 pt]
&=\frac{1}{\abs{U_t}}\left(\int_t^1\int_{\partial U_s}\frac{x}{\abs{\nabla u(x)}}\,d\Hn(x)\,ds+\int_K x\,d\Ln(x)\right),
\end{split}
\]
and we can infer that $x(t)$ is an absolutely continuous function, since $\abs{\nabla u}>0$ implies that $\abs{U_t}$ is an absolutely continuous function as well. Moreover, on $\partial U_t$ we have that for every $\nu\in\mathbb{S}^{n-1}$, 
\begin{equation}
\label{eq: levelvalue}
u(x(t)+r(t)\nu)=t,
\end{equation}
from which 
\begin{equation}
\label{eq: gradientlevel}
\nabla u(x(t)+r(t)\nu)=-C_t\nu.
\end{equation}
Differentiating \eqref{eq: levelvalue}, and using \eqref{eq: gradientlevel}, we obtain
\begin{equation}
\label{eq: derlevelvalue}
-C_t\, x'(t)\cdot \nu-C_t \,r'(t)=1.
\end{equation}
Finally, joining \eqref{eq: derlevelvalue} and \eqref{eq: r'}, and the fact that $\abs{\nabla u}=C_t$ on $\partial U_t$, we get
\[
x'(t)\cdot \nu = 0
\]
for every $\nu\in\mathbb{S}^{n-1}$, so that $x(t)$ is constant and $U_t$ are concentric balls for a.e. $t\in(0,1)$. In particular, $\Set{u=1}=\bigcap_t U_t$ and $\Set{u>0}=\bigcup_t U_t$ share the same center.
\end{proof}

\printbibliography[heading=bibintoc]
\Addresses
\end{document}